\documentclass[11pt,a4paper]{article}
\usepackage[utf8x]{inputenc}
\usepackage{ucs}
\usepackage{amsmath}
\usepackage{amsfonts}
\usepackage{amssymb}
\usepackage{amsthm}
\usepackage{color}
\usepackage{dsfont}
\usepackage{geometry}

\newtheorem{defn}{Definition}[section]
\newtheorem{thm}{Theorem}[section]
\newtheorem{prop}{Proposition}[section]

\newtheorem{rmk}{Remark}[section]
\newtheorem{lem}{Lemma}[section]

\title{Almost sure behavior of the critical points \\of random polynomials}
\author{J\"urgen Angst$^1$, Dominique Malicet$^2$ and Guillaume Poly$^1$}

\newcommand\blfootnote[1]{%
  \begingroup
  \renewcommand\thefootnote{}\footnote{#1}%
  \addtocounter{footnote}{-1}%
  \endgroup
}

\begin{document}

\maketitle

\begin{abstract}
Let $(Z_k)_{k\geq 1}$ be a sequence of independent and identically distributed complex random variables with common distribution $\mu$ and let $P_n(X):=\prod_{k=1}^n (X-Z_k)$ the associated random polynomial in $\mathbb C[X]$. In \cite{MR3283656}, the author established the conjecture stated by Pemantle and Rivin in \cite{MR3363974} that the empirical measure $\nu_n$ associated with the critical points of $P_n$ converges weakly in probability to the base measure $\mu$. In this note, we establish that the convergence in fact holds in the almost sure sense. Our result positively answers a question raised by Z. Kabluchko and formalized as a conjecture in the recent paper \cite{Michelen}.
\end{abstract}


\blfootnote{This work was supported by the ANR grant UNIRANDOM, ANR-17-CE40-0008.} 
\blfootnote{$^1$ Univ Rennes, CNRS, IRMAR - UMR 6625, F-35000 Rennes, France.} 
\blfootnote{$^2$ Universit\'e Gustave Eiffel, CNRS, LAMA - UMR 8050, F-77420 Champs-sur-Marne, France.} 
\blfootnote{email: jurgen.angst@univ-rennes.fr, dominique.malicet@univ-eiffel.fr, guillaume.poly@univ-rennes.fr}

\section{Introduction and main result}

We consider complex random polynomials whose roots are given by independent and identically distributed random variables. Namely, if $(Z_k)_{k\geq 1}$ is a sequence of i.i.d. complex random variables with distribution $\mu$, defined on a common probability space $(\Omega, \mathcal F, \mathbb P)$, we consider the random polynomial $P_n$ in $\mathbb C[X]$ defined by 
\[
P_n(X):=\prod_{k=1}^n (X-Z_k).
\]
This model was proposed by Pemantle and Rivin in \cite{MR3363974} where they raise the following question: for $n$ large, is the distribution of the critical points of $P_n$ always close to the distribution of its roots, in some stochastic sense? They illustrate this question with some examples and partial results.  
Since then, this model has attracted a lot of attention, see for example \cite{MR2970701, MR3283656,MR3698743,MR3896083,MR3940764,MR4136480,MR4474893} and the references therein.

\newpage
A decisive step on this problem was done by Kabluchko in \cite{MR3283656} where he positively answered the question in the most general framework. To formalize the statement, let us introduce $W_1,\ldots, W_{n-1}$ the (random) critical points  of $P_n$, i.e. the zeros of $P_n'$  (with possible multiple occurences) and the (random) empirical measures
\[
\mu_n= \mu_n^{\omega}:=\frac{1}{n} \sum_{i=1}^n \delta_{Z_i(\omega)}, \qquad \nu_n=\nu_n^{\omega}:= \frac{1}{n-1} \sum_{i=1}^{n-1} \delta_{W_i(\omega)}.
\]
The general result obtained by Kabluchko is the following.
\begin{thm}[Theorem 1.1 of \cite{MR3283656}] \label{thm.kab}
As $n$ goes to infinity, the sequence of empirical measures $(\nu_n)$ converges in probability to $\mu$, in the space of complex probability measures, equipped with the topology of the convergence in distribution. 
\end{thm}
This unconditional statement extends anterior results where the convergence is established under various assumptions on $\mu$, for example in \cite{MR3363974} in the case where $\mu$ has finite energy and \cite{MR2970701} in the case where $\mu$ is supported in the circle. Though very general and satisfying, the above statement can potentially be further improved. Indeed, the convergence of $(\nu_n)_n$ to $\mu$ is here stated in probability (as in most of the results of the papers cited above),  however one could expect that it holds for a fixed typical realization of the random variables. The objective of the paper is precisely to establish this almost sure convergence, without any additional assumption. 

\begin{thm}\label{thm.main}
Almost surely with respect to $\mathbb P$, the sequence of empirical measures $(\nu_n)$ converges in distribution to $\mu$ as $n$ goes to infinity. 
\end{thm}

Note that the question of the almost sure convergence of the empirical measure of critical points was in fact raised by Z. Kabluchko himself and it is in particular formalized as Conjecture 4.3 at the very end of the recent paper \cite{Michelen}.
\begin{rmk}
Note that following the reference \cite{MR2236634}, the statement of Theorem \ref{thm.main} is equivalent to the fact that for all bounded continuous test functions $\varphi$,
\[
\int_{\mathbb C} \varphi(x)d\nu_n(x) \xrightarrow[n \to +\infty]{a.s.} \int_{\mathbb C} \varphi(x)d\mu(x).
\]
Similarly, the statement of Theorem \ref {thm.kab} is equivalent to the fact that for all bounded continuous test functions $\varphi$, we have the following convergence in probability
\[
\int_{\mathbb C} \varphi(x)d\nu_n(x) \xrightarrow[n \to +\infty]{\mathbb P} \int_{\mathbb C} \varphi(x)d\mu(x).
\]
\end{rmk}

Our general strategy to establish Theorem \ref{thm.main} relies on the celebrated \textit{Jensen's formula}, which roughly speaking relates the zeros, the poles and the growth of a meromorphic function. We want to apply the formula to the logarithmic derivative $S_n:=P_n'/P_n$, whose poles are the roots of $P_n$ and zeros are the critical points of $P_n$. This will enable us to compare the distributions $\mu_n$ and $\nu_n$, provided that we have some almost sure control of the growth of $|S_n|$, by above and by below. The upper bound is rather easy to obtain, but the lower bound is more subtle and relies on a general anti-concentration result for sums of random vectors, namely a multidimensional version of \textit{Kolmogorov--Rogozin inequality}. This approach allows to make explicit a crucial one-sided estimate between $\mu_n$ and $\nu_n$ which, as we shall see below, is sufficient to conclude. \\

The paper is organized as follows. In Section \ref{sec.discrete}, for technical reasons, we first establish Theorem \ref{thm.main} in the elementary case where the base measure $\mu$ has finite support. In Section \ref{sec.concentration}, we state and prove the anti-concentration estimate mentioned above. In Section \ref{sec.estimates}, we then deduce some controls of the growth of $\frac{P_n'}{P_n}$ that hold almost surely. Finally in Section \ref{sec.convergence}, we combine Jensen's Formula with the estimates previously obtained to prove Theorem \ref{thm.main} in the general case.\\

We keep for the whole paper the notations that we introduced above for the random variables $Z_n$, the polynomials $P_n$ and the associated distributions $\mu$, $\mu_n$, $\nu_n$. Let us also introduce some additional notations:

\begin{itemize}
	\item We denote by $\lambda_{\mathbb C}$ (resp. $\lambda_{\mathbb R}$) the Lebesgue measure on $\mathbb C$ (resp $\mathbb R$). We will say that a set is of full measure if its complement is negligible i.e of measure zero.
	\item We denote by $C(a,r)$ (resp. $D(a,r)$, $\overline{D}(a,r)$) the circle (resp. open disk, closed disk) of center $a$ and radius $r$. In the case where $a=0$ and $r=1$, we simply denote $C$ (resp. $D$, $\overline{D}$).
		\item We will denote by $||x||:=(\sum_{i=1}^d |x_i|^2)^{1/2}$ the standard Euclidean norm of a given vector $x=(x_1, \ldots, x_n)$ in $\mathbb R^d$ or $\mathbb C^d$. In the real case, the Euclidean scalar product will be denoted by $x \cdot y :=\sum_{i=1}^d x_i y_i$.
	\item If $K \subset \mathbb C$ is compact and $f:K\rightarrow \mathbb C$ is continuous, we set $||f||_K:=\sup_{z\in K} |f(z)|$.
	\item We denote by $\log^+$ and $\log^-$ the positive and negative parts of the standard logarithm function, i.e. for $x>0$,
	\[
	\log^+(x):=\log(x)\mathds{1}_{x>1}, \qquad \log^-(x):=- \log(x)\mathds{1}_{x<1}.
	\]
\end{itemize}

\section{Proof of the main result}\label{sec.proofs}
The rest of the paper is devoted to the proof of our main result i.e. Theorem \ref{thm.main} stated in the introduction. For technical reasons, we first give the proof in the elementary case where the base measure $\mu$ has finite support. The proof in the general case is the object of the next Sections \ref{sec.concentration}, \ref{sec.estimates} and \ref{sec.convergence}.

\subsection{The case of a measure with finite support}\label{sec.discrete}
In this section, we prove Theorem \ref{thm.main} in the special and simpler case where the measure $\mu$ has a finite support $\text{supp}(\mu):=\{z_1, \dots, z_r\}$ with $p_i:=\mu(\{z_i\})>0$ and $\sum_{i=1}^r p_i=1$. In that case, the random polynomial $P_n$ has the form 
\[
P_n(X)= \prod_{i=1}^r (X-z_i)^{N_i},
\]
where $N=(N_1, \ldots, N_r)$ has multinomial distribution with parameters $n$ and $(p_i)_{1 \leq i \leq r}$, namely for $k=(k_1, \ldots, k_r)$ such that $\sum_{i=1}^r k_i = n$ we have
\[
\mathbb P( N=k) =\frac{n!}{\prod_{i=1}^r k_i!} p_i^{k_i}.
\]
The empirical measure then reads 
$
\mu_n:= \sum_{i=1}^r \frac{N_i}{n} \delta_{z_i},
$
and by the law strong of large numbers, $\mathbb P$-almost surely,  we have 
\[
\left( \frac{N_1}{n}, \ldots, \frac{N_r}{n} \right) \xrightarrow[n \to +\infty]{a.s.} (p_1, \ldots, p_r).
\]
Therefore, $\mathbb P$-almost surely, $\mu_n$ converges weakly to $\mu$. Now the derivative $P_n'$ is given by 
\[
P_n'(X) =\sum_{i=1}^r N_i (X-z_i)^{N_i-1} \prod_{\substack{1\leq j \leq r \\ j\neq i}} (X-z_j)^{N_j}= \left( \prod_{i=1}^r (X-z_i)^{N_i-1}\right) Q_r(X),
\]
where $Q_r$ is the following polynomial of degree $r-1$
\[
Q_r(X):=\sum_{i=1}^r N_i \prod_{\substack{1\leq j \leq r \\ j\neq i}} (X-z_j).
\]
If $z_1', \ldots, z_{r-1}'$ denote the (random) complex zeros of $Q_r$ (with possible multiple occurrences), the empirical measure $\nu_n$ of the critical points of $P_n$ thus reads
\[
\nu_n=\sum_{i=1}^r \frac{N_i-1}{n-1} \delta_{z_i} + \frac{1}{n-1} \sum_{i=1}^{r-1} \delta_{z_i'}.
\]
As a result, again by the strong law of large numbers, $\mathbb P$-almost surely,  $(\nu_n)$ converges weakly to $\mu$.

\subsection{A general anti-concentration estimate}\label{sec.concentration}
\noindent
As mentioned in the introduction, our global strategy in order to establish Theorem \ref{thm.main} in the general case requires establishing some almost sure controls for the logarithmic derivative 
\[
S_n(z):=\frac{P_n'(z)}{P_n(z)}=\sum_{k=1}^n \frac{1}{z-Z_k}, \;\; z \in \mathbb C.
\] 
Lower bounding $|S_n(z)|$ then amounts to establish an anti-concentration estimate for the sum just above. Note that since we do not impose any condition of the common distribution $\mu$ of the variables $Z_k$, we cannot impose any condition on the summands $1/(z-Z_k)$. Therefore, the goal of this section is to establish an anti-concentration estimate for sums of i.i.d. random vectors without any assumption on the common underlying distribution. The result we obtain, namely Proposition \ref{prop.kolmorogo} below, is a multidimensional version of Kolmogorov--Rogozin inequality, see e.g. Theorem 2.22 on p. 76 in \cite{MR1353441}.  This kind of estimate is not new, some variants appear for example in the seminal papers \cite{MR14626} or \cite{MR295403} under finite moments conditions. Nevertheless, to the best of our knowledge, the unconditional statement that we need does not appear in the literature, therefore we give here a detailed proof. Let us first recall the following standard definition.

\begin{defn}\label{def.degenere}
A random vector $X=(X_1,\ldots,X_d)$ with values in $\mathbb R^d$ (resp. $\mathbb C^d$) is non-degenerate if there does not exists a non trivial linear combination of $X_1,\ldots,X_d$ which is almost surely constant. Equivalently, it means that the random variables $X_1,\ldots,X_d,1$ are linearly independent on $\mathbb R$ (resp. $\mathbb C$).
\end{defn}

\begin{rmk}In the real case, if the random vector $X=(X_1,\ldots,X_d)$ is square integrable with covariance matrix $K$, then it is non-degenerate iff  $K$ is positive definite. Indeed, for any $\lambda=(\lambda_1, \ldots, \lambda_d) \in \mathbb R^d$ with transpose $\lambda^*$, we have $\text{var}(\lambda \cdot X)= \lambda K \lambda^*$.
\end{rmk}

Under the above non-degeneracy assumption, we can now state the following general anti-concentration estimate, whose main interest is to provide a better upper bound as the dimension increases. 

\begin{prop}\label{prop.kolmorogo}
	Let us consider $(X^n)_{n \geq 1}=(X_1^n, \ldots, X_d^n)_{n \geq 1}$ a sequence of i.i.d. non-degenerate random vectors with values in $\mathbb C^d$ and set $S_n:=\sum_{k=1}^n X^k$. Then, there exists a positive constant $C$ which depends on ${d,r}$ and the law of $X$, but is independent of $n$ such that 
	\[
	\sup_{x \in \mathbb C^d} \mathbb P\left( ||S_n-x|| \leq r \right) \leq \frac{C}{n^{d/2}}.
	\]
\end{prop}

\begin{rmk}
The important point here is that no integrability assumption is required. For $d=1$, it gives back the i.i.d. case of Kolmogorov--Rogozin inequality, which was used in \cite{MR3283656} to prove Theorem \ref{thm.kab}. However, the one-dimensional version of the inequality appears to be not sufficient to prove the almost sure version of the theorem, mainly because one cannot apply some Borel--Cantelli type argument due to the divergence of the series $\sum_n n^{-1/2}$. As we shall see in Section \ref{sec.estimates} below, enlarging the dimension allows to bypass this difficulty.
\end{rmk}
	
\begin{proof}[Proof of Proposition \ref{prop.kolmorogo}]
Let us give a detailed proof of the above anti-concentration estimate. 	We start with a simple lemma, which is a multidimensional version of Lemma 1.5 p. 14 of \cite{MR1353441}.

\begin{lem}\label{lem.petrov}
Let $X$ be a non-degenerate random vector in $\mathbb R^d$. Then, there exists positive constants $\delta>0$ and $\kappa>0$ such that, for all $t \in \mathbb R^d$ such that $||t||\leq \delta$, we have 
\[
\left|   \mathbb E\left[ e^{i t \cdot X} \right]\right| \leq 1- \kappa ||t||^2 \leq e^{-\frac{\kappa}{2}  ||t||^2 }.
\]
\end{lem}
\begin{proof}[Proof of Lemma \ref{lem.petrov}]
Let us first consider the case where $X=(X_1, \ldots, X_d)$ is non-degenerate and square integrable, with covariance matrix $K_{ij}:=\text{cov}( X_i, Y_j)$. Then, its characteristic function is twice differentiable and by Taylor formula, there exists $\delta>0$ such that for $t=(t_1, \ldots, t_d) \in \mathbb R^d$ with $||t||\leq \delta$, if $t^*$ denotes its transpose 
\[
\left|   \mathbb E\left[ e^{i t \cdot X} \right]\right| \leq 1- \frac{1}{4} t K t^* \leq 1- \kappa ||t||^2 \leq e^{-\frac{\kappa}{2}  ||t||^2},
\]
where $4\kappa$ can be chosen e.g. as the smallest (positive) eigenvalue of the (positive definite) covariance matrix $K$.
Now, if $X=(X_1, \ldots, X_d)$ is a non-degenerate random vector which is not square integrable, let us choose $R$ large enough so that
\[
c_R=\mathbb P( ||X||\leq R) >0.
\]
One can then consider the vector $Y_R$ with conditional distribution $\mathcal L(X \; | \ ||X|| \leq R)$, i.e.  for all bounded measurable function $h$
\[
\mathbb E[ h(Y_R)] =\frac{1}{c_R} \times \mathbb E\left[ h(X) \mathds{1}_{|| X|| \leq R} \right].
\]
Up to enlarging $R$, we can further assume that the vector $Y_R$ is non-degenerate and it is bounded thus square integrable. By the first part of the proof, there exists a positive constants $\delta_R$ and $\kappa_R$ such that, for all $||t||\leq \delta_R$
\[
\left|   \mathbb E\left[ e^{i t \cdot Y_R} \right]\right| \leq 1- \kappa_R ||t||^2.
\]
By the triangular inequality, we have then
\[
\begin{array}{ll}
\displaystyle{\left|   \mathbb E\left[ e^{i t \cdot X} \right]\right|} & \displaystyle{ \leq \left|  \mathbb E\left[ e^{i t \cdot X} \mathds{1}_{||X|| \leq R} \right]\right| +\left|  \mathbb E\left[ e^{i t \cdot X} \mathds{1}_{||X|| > R} \right]\right| } \\
\\
& \leq  \displaystyle{c_R \left|   \mathbb E\left[ e^{i t \cdot Y_R} \right]\right| +\mathbb P( ||X||>R) }\\
\\
& \leq  \displaystyle{c_R \left( 1- \kappa_R ||t||^2 \right) +1-c_R }\\
\\
& = \displaystyle{1- c_R  \kappa_R ||t||^2 \leq \exp\left( -\frac{c_R  \kappa_R}{2} ||t||^2\right)}.
\end{array}
\]
\end{proof}
Combining the above estimate with Esseen inequality on the concentration function, we can now complete the proof of Proposition \ref{prop.kolmorogo}, first in the case where the vector $X$ takes values in $\mathbb R^d$.
 Let us recall the classical Esseen concentration inequality, see e.g. Lemma 7.17 of \cite{MR2289012}, if $X$ is a random vector in $\mathbb R^d$, then for any ${r>0}$ and $ {\varepsilon > 0}$, 
\[
\displaystyle \sup_{x \in {\mathbb R}^d} \mathbb P ( ||X - x|| \leq r ) \leq \kappa_{d,\varepsilon} r^d \int_{\substack{t \in {\mathbb R}^d\\ ||t|| \leq \varepsilon/r}}  \left|\mathbb E\left[  e^{it \cdot X}\right]\right| dt,
\]
for some constant ${\kappa_{d,\varepsilon}}$ depending only on ${d}$ and ${\varepsilon}$.
In particular, uniformly in $x\in \mathbb R^d$
\begin{equation}\label{eq.esseen}
\displaystyle \mathbb P ( || S_n-x|| \leq r ) \leq \kappa_{d,\varepsilon} r^d \int_{\substack{t \in {\mathbb R}^d\\ ||t|| \leq \varepsilon/r}}  \left|\mathbb E\left[  e^{it \cdot S_n }\right]\right| dt.
\end{equation}
By Lemma \ref{lem.petrov} and by independence of the vectors $X^k$, there exists positive constants $\delta>0$ and $\kappa>0$ such that for $||t|| \leq \delta$
\[
\left|\mathbb E\left[  e^{it \cdot S_n}\right]\right|\leq e^{-\frac{n \kappa}{2} ||t||^2}.
\]
Therefore, choosing $\varepsilon$ small enough so that $\varepsilon/r = \delta$, and injecting this estimate in Equation \eqref{eq.esseen}, we get 
\[
\displaystyle \mathbb P ( || S_n-x|| \leq r ) \leq \kappa_{d,\varepsilon} r^d \int_{\substack{t \in {\mathbb R}^d\\ ||t|| \leq \delta}}  e^{-\frac{n \kappa}{2} ||t||^2}dt.
\]
Performing the change of variables $t \to t/\sqrt{n}$ in $\mathbb R^d$, we finally get 
\[
\displaystyle \mathbb P ( || S_n-x|| \leq r ) \leq \frac{\kappa_{d,\varepsilon} r^d}{n^{d/2}} \int_{\substack{t \in {\mathbb R}^d\\ ||t|| \leq \sqrt{n} \delta}}  e^{-\frac{\kappa}{2} ||t||^2} dt \leq \frac{\kappa_{d,\varepsilon} r^d}{n^{d/2}} \int_{t \in {\mathbb R}^d}  e^{-\frac{\kappa}{2} ||t||^2} dt, 
\]
hence the result with the constant 
\[
C_{d,r}:=\kappa_{d,\frac{\delta}{r}} \times r^d \times \int_{t \in {\mathbb R}^d}  e^{-\frac{\kappa}{2} ||t||^2} dt<+\infty.
\]

This establishes Proposition \ref{prop.kolmorogo} in the case where if $X$ takes values in $\mathbb R^d$. Let us now complete the proof in the complex case. We first note the following general fact.

\begin{lem}\label{lem.complexification}
Let $X=(X_1,\ldots,X_d)$ be a random vector with values in $\mathbb C^d$. Then, the two following assertions are equivalent.
\begin{enumerate}
	\item The marginals $X_1,\ldots,X_d$ are linearly dependent on $\mathbb C$.
	\item For every $a$, $b$ in $\mathbb R$, the real variables $a\Re(X_1)+b\Im(X_1),\ldots,a\Re(X_d)+b\Im(X_d)$ are linearly dependent on $\mathbb R$.
\end{enumerate}
\end{lem} 
\begin{proof}
The direct assertion $1\Rightarrow 2$ is clear, so we are left to establish the converse $2\Rightarrow 1$. We use the following characterization: the coordinates of a random vector $Y$ with values in $\mathbb C^d$ (resp. $\mathbb R^d$) are linearly dependent on $\mathbb C$ (resp. $\mathbb R$) if and only if $\det(Y^1,\ldots,Y^d)=0$ a.s., where $Y^1,\ldots, Y^d$ are independent copies of $Y$. Indeed if there is a non trivial linear relation between the coordinates of $Y$, then it is also a linear relation between the coordinates of $Y^i$ and we deduce that the determinant is zero. Conversely if the latter is null, we expand the determinant with respect to the last column to obtain a linear relation between the coordinates of $Y^d$, whose coefficients are random minor determinants of size $(d-1)\times (d-1)$, independent from $Y^d$. For each of these minors, two cases can occur: either the minor is not almost surely $0$ and so we obtain a non trivial linear relation between the coordinates of $Y^d$, and so of $Y$ ; or else the minor is $0$ a.s. and we deduce by induction on $d$ that there exists a non trivial linear relation between some of the coordinates of $Y$.
	
Let us now apply the above characterization of non-degeneracy in terms of determinant to the random vector $Y(t)=(Y_1(t),\ldots,Y_d(t))$ where $Y_k(t)=\Re(X_k)+t\Im(X_k)$ and $t$ is a number in $\mathbb R$ or $\mathbb C$. So let $Y^1(t),\ldots, Y^d(t)$ be random independent copies of $Y(t)$ and $D(t)=\det(Y^1(t),\ldots,Y^d(t))$. If we assume the second assertion above, then for every $t$ in $\mathbb R$, $D(t)=0$ a.s., but since $D$ is polynomial in $t$ we deduce that almost surely, $D$ is the zero polynomial. In particular, we have $D(i)=0$ where $i^2=-1$. Since $Y(i)=X$, we deduce the first assumption.
\end{proof}

With this general lemma in hand, we can indeed establish Proposition \ref{prop.kolmorogo} in the remaining case where $X^n=(X_1^n,\ldots,X_d^n)$ takes values in $\mathbb C^d$. Indeed by applying the last Lemma \ref{lem.complexification} to $(X_1^n,\ldots,X_d^n,1)$, we deduce that there exists real numbers $a$ and $b$ so that the vector $Y^n=a\Re(X^n)+b\Im(X^n)$ is non-degenerate, and we can assume $a^2+b^2=1$. Then,  given $x \in \mathbb C^d$ and setting $y:=a\Re(x)+b\Im(x)$ and $T_n:=\sum_{k=1}^n Y^k$, we have the upper bound $||T_n-y||\leq ||S_n-x||$. As a result, the upper bound for $\mathbb P\left( ||T_n-y|| \leq r \right)$ established in the real case in the first part of the proof, also allows to upper bound the corresponding probability $\mathbb P\left( ||S_n-x|| \leq r \right)$ in the complex case.
\end{proof}

\subsection{Estimates on the logarithmic derivative} \label{sec.estimates}

As announced above, we now give some almost sure lower and upper bounds for the logarithmic derivative $S_n=P_n'/P_n$.
The easiest part is to upper bound $|S_n|$. Namely, in the next Lemma \ref{lem.major}, we obtain a rather direct uniform bound on ``generic" circles. The lower bound is more subtle and requires more attention. The estimate we obtain in Lemma \ref{lem.minor} below is weaker and it is not uniform, but still,  it will be sufficient for our purpose. 

\if{
It is based on the anti-concentration estimate establish in Section \ref{sec.concentration} above, noticing that for generic complex numbers $(z_i)_{1\leq i \leq 3}$ and any random variable $Z$ whose support is not finite, the complex random vector $\left( \frac{1}{z_1-Z},\frac{1}{z_2-Z},\frac{1}{z_3-Z}\right)$ is automatically non-degenerate.
}\fi

\begin{lem}\label{lem.major}
Almost surely with respect to $\mathbb P$, there exists a set of $\lambda_{\mathbb C}\otimes \lambda_{\mathbb R}$-full measure of couples $(r,a)\in \mathbb R^+\times \mathbb C$ such that, as $n$ tends to infinity
$$\log^+||S_n||_{C(a,r)}=O(\log(n)).$$
\end{lem}

\begin{proof}
Let us first consider the case $a=0$.
Since 
\[
\int_{-\infty}^{+\infty}\mathbb E\left[\frac{1}{|r-|Z_1||^{\frac{1}{2}}}\mathds{1}_{|r-|Z_0||\leq 1}\right]dr\leq \int_{-1}^{1}\frac{1}{|r|^{\frac{1}{2}}}dr<+\infty,
\]
we have for $\lambda_{\mathbb R}$-almost every $r>0$ that 
\[
\mathbb E\left[\frac{1}{|r-|Z_1||^{\frac{1}{2}}}\right]= \mathbb E\left[\frac{1}{|r-|Z_1||^{\frac{1}{2}}}\mathds{1}_{|r-|Z_0||\leq 1}\right]+\underbrace{\mathbb E\left[\frac{1}{|r-|Z_1||^{\frac{1}{2}}}\mathds{1}_{|r-|Z_0||> 1}\right]}_{\leq 1}<+\infty.
\]
For such a number $r>0$, we have by the law of large numbers that $\mathbb P$-almost surely 
$$\sum_{k=1}^n \frac{1}{|r-|Z_k||^{\frac{1}{2}}}=O(n).$$ 
Then, we can upper the supremum norm on the circle as follows
$$\begin{array}{ll}\displaystyle||S_n||_{C(0,r)}&\displaystyle\leq \sup\left\{\sum_{k=1}^n\frac{1}{|z-Z_k|}, z\in C(0,r)\right\}
	\displaystyle\leq \sum_{k=1}^n\frac{1}{|r-|Z_k||}\\
	&\displaystyle\leq n \sup_{1\leq k \leq n}\frac{1}{|r-|Z_k||}
\displaystyle\leq n \left(\sum_{k=1}^n \frac{1}{|r-|Z_k||^{\frac{1}{2}}}\right)^2 =O(n^3),
	\end{array}$$
and we deduce that $\log^+ ||S_n||_{C(0,r)}=O(\log n)$. Next, for any $a$ in $\mathbb C$, by applying the above reasoning to $Z_n'=Z_n-a$, we obtain that $\mathbb P$-almost surely, for $\lambda_{\mathbb R}$-almost every $r>0$, $\log^+ ||S_n||_{C(a,r)}=O(\log n)$. 
Finally, let us define the set
\[
E:=\{(\omega,a,r)\in \Omega\times \mathbb C\times \mathbb R_+| \log^+ ||S_n||_{C(a,r)}=O(\log n)\}.\]
We have just proved that for any $a \in \mathbb C$, for $\lambda_{\mathbb R}$-almost all $r>0$ and for $\mathbb P$-almost all $\omega$, we have $(\omega, a,r) \in E$. By Fubini--Tonelli theorem, we then conclude that for $\mathbb P$-almost all $\omega$ in $\Omega$, for $\lambda_{\mathbb C}$-almost all $a\in \mathbb C$ and $\lambda_{\mathbb R}$-almost all $r>0$, $(\omega, a,r) \in E$.
\end{proof}

\begin{lem}\label{lem.minor}
Let us assume here that the base measure $\mu$ has not a finite support. Then $\mathbb P$-almost surely, there exists a set of $\lambda_{\mathbb C}\otimes \lambda_{\mathbb C}\otimes \lambda_{\mathbb C}$-full measure of triplets $(z_1,z_2,z_3) \in \mathbb C^3$ such that, for all $n\geq n_0$ large enough, for at least one of the three complex numbers $z_1,z_2,z_3$ the inequality $|S_n(z_i)| \geq 1$ holds (with the convention $|S_n(z)|=+\infty$ if $z$ is a pole of $S_n$).
\end{lem}

\begin{proof}[Proof of Lemma \ref{lem.minor}]
Let  $z_1,z_2,z_3$ be three pairwise distinct complex numbers and let $Z$ be a random variable with distribution $\mu$. Let us consider the random vector  
\[
V:=\left( \frac{1}{z_1-Z},\frac{1}{z_2-Z},\frac{1}{z_3-Z}\right) \in \mathbb C^3.
\]
We claim that $V$ is non degenerated in the sense of Definition \ref{def.degenere}. Indeed, if there exists complex numbers $\alpha:=(\alpha_1,\alpha_2,\alpha_3) \in \mathbb C^3\backslash\{0\}$ and $\beta\in \mathbb C$ such that  
 \[
 \frac{\alpha_1}{z_1-Z} +  \frac{\alpha_2}{z_2-Z}+ \frac{\alpha_3}{z_3-Z} + \beta=0,
 \]
then $Z$ lives in the set of roots of a non zero rational function and it contradicts that the support of $\mu$ is infinite.
Thus, we can apply Proposition \ref{prop.kolmorogo}. Namely, if $(Z_k)_{k\in\mathbb N}$ is a sequence of i.i.d. random variables with distribution $\mu$, and if $V^k$ is defined as 
\[
V^k:=\left( \frac{1}{z_1-Z_k},\frac{1}{z_2-Z_k},\frac{1}{z_3-Z_k}\right) \in \mathbb C^3,
\]
so that
\[
(S_n(z_1),S_n(z_2),S_n(z_3)) =\sum_{k=1}^n V^k,
\]
the anti-concentration estimate given by Proposition \ref{prop.kolmorogo} ensures that
\[
\mathbb P( || (S_n(z_1),S_n(z_2),S_n(z_3)|| \leq \sqrt{3}) = O \left(  \frac{1}{n^{3/2}}\right).
\]
Therefore, by Borel--Cantelli Lemma, we deduce that $\mathbb P$-almost surely, i.e. on a set of $\mathbb P$-full measure which depends on the triplet $(z_1,z_2,z_3)$, for $n$ large enough, we have
$|| (S_n(z_1),S_n(z_2),S_n(z_3)|| \geq  \sqrt{3}$ and thus at least of the $|S_n(z_i)|$ is larger than $1$. 
\par
\medskip
At this point, we have thus proved that for $z_1,z_2,z_3$ pairwise distinct complex numbers, $\mathbb P$-almost surely, for $n$ large enough, we have $|S_n(z_i)|\geq 1$ for $i=1$, $2$ or $3$. We conclude as in previous lemma with Fubini--Tonelli Theorem. Precisely, since the set of pairwise distinct triplets of $\mathbb{C}^3$ has $\lambda_{\mathbb C}\otimes\lambda_{\mathbb C}\otimes \lambda_{\mathbb C}$-full measure, what we established above implies that the set
\[
E:=\{ (\omega,z_1,z_2,z_3) \in \Omega\times\mathbb  C^3, \, (\forall i\in \{1,2,3\}, |S_n^{\omega}(z_i)| < 1 )\; \text{for infinitely many} \; n\in \mathbb N \},
\]
is $\mathbb P\otimes\lambda_{\mathbb C}\otimes \lambda_{\mathbb C}\otimes \lambda_{\mathbb C}$-negligible in $\Omega\times \mathbb C\times \mathbb C\times \mathbb C$, and we conclude.

\end{proof}

\subsection{Proof in the general case} \label{sec.convergence} 

We can now complete the proof of our main Theorem \ref{thm.main} under the assumption that the base measure $\mu$ has not a finite support, since this finite support case has been treated appart in Section \ref{sec.discrete}. As briefly mentioned in the introduction, the starting point of our proof relies on Jensen's formula. Its standard version says that if ${f}$ is a meromorphic function on a disk $\overline{D}(a,r)$ without pole or zero at the point $a$, then
\[
\displaystyle \log |f(a)| - \int_0^1 \log |f(a+re^{2\pi i t})|\ dt = \sum_{\rho\in D(a,r)} \log \frac{|\rho-a|}{r} \displaystyle - \sum_{\zeta\in D(a,r)} \log \frac{|\zeta-a|}{r},
\]
where ${\rho}$ and ${\zeta}$ range over the zeros and poles of ${f}$ respectively (counting the multiplicity). The formula thus allows to compare, with logarithmic weights, the numbers of zeros and poles of the function in this disk. In order to compare the distributions $\mu_n$ and $\nu_n$, a very  natural idea is therefore to apply Jensen's formula in a ``generic disk'' in the complex plane, to the logarithmic derivative $S_n=P_n'/P_n$.
\par
\smallskip
An easy fact to verify is that the Jensen's formula on any disk of $\mathbb C$ can be deduced from Jensen's formula on the unit disk $D$, by applying affine transformations $z\mapsto \alpha z+\beta$. For our purpose, we will also need estimates at the neighborhood of the infinity, so we actually want estimates on any disk of the Riemann sphere $\widehat{\mathbb C}=\mathbb C\cup\{\infty\}$. To do this, we will replace the affine transformations by Möbius transformations $z\mapsto \frac{\alpha z+\beta}{\gamma z+\delta}$. The discussion above motivates the introduction of the following new notations:
	\begin{itemize} 
		\item We denote by $\mathcal{A}=\{z\mapsto \alpha z+\beta| \alpha,\beta\in\mathbb{C}, \alpha\not=0\}$  the set of invertible affine transformations of $\mathbb C$. We endow $\mathcal A$ with the measure $\lambda_{\mathcal A}$ inhereted from the Lebesgue measure $\lambda_{\mathbb C}\otimes \lambda_{\mathbb C}$ on $\mathbb C^2$.
		\item We denote by $\mathcal{M}=\{z\mapsto \frac{\alpha z+\beta}{\gamma z+\delta}| \alpha,\beta,\gamma,\delta \in\mathbb{C}, \alpha\delta-\beta\gamma\not=0\}$  the set of invertible Möbius transformations of $\mathbb C$. We endow $\mathcal M$ with the measure $\lambda_{\mathcal M}$ inherated from the Lebesgue measure $\lambda_{\mathbb C}\otimes \lambda_{\mathbb C}\otimes \lambda_{\mathbb C}\otimes \lambda_{\mathbb C}$ on $\mathbb C^4$.
	\end{itemize}
Our choices of measures on both sets $\mathcal A$ and $\mathcal M$ are not very canonical. It might indeed be more natural to use the natural Haar measures on these geometric structures, but the above measures will only serve us to define sets of full measure and negligible, hence our simple choice of product measures.
\par
\smallskip
The key estimate derived from Jensen's inequality that we will use in the sequel is the following inequality. Recall that $C$ denotes the unit circle in $\mathbb C$.

\begin{prop}\label{prop.Jensen}
Let $P$ be a complex polynomial, let $\mathcal Z$ be its set of roots, $\mathcal C$ its set of critical points and let $S=\frac{P'}{P}$ be its logarithmic derivative. Fix a Möbius transformation $u$ in $\mathcal M$ and denote by $a=u^{-1}(0)$ and $C'=u^{-1}(C)$. Then, provided $a$ is not a zero or a pole of $S$, we have
\begin{equation}
\sum_{\rho\in \mathcal C} \log^-|u(\rho)| \displaystyle - \sum_{\zeta \in \mathcal Z} \log^-|u(\zeta)| \leq \log \|S\|_{C'} - \log |S(a)|,
\label{eq.upjensen}
\end{equation}
	where the elements of  $\mathcal C$ and $\mathcal Z$ are counted with multiplicities.
\end{prop}
\begin{rmk}~Let us make a few remarks on the last statement.
	\begin{enumerate}
		\item At first sight, it might not be clear that in the inequality \eqref{eq.upjensen}, we only count the roots and critical points of $P$ belonging to the disk $C'$, but it is implicit since the function $\log^-|u(\cdot)|$ is null outside this disk.
		\item Let us emphasize that the inequality \eqref{eq.upjensen} is not stated in modulus, it only provides an upper bound. As such, it only allows to bound by above, up to a remainder term, the distribution of the critical points of $P$ by the distribution of the roots of $P$ applied to a certain class of test functions. As we shall see below, this one-sided estimate will be still sufficient for our purpose.
	\end{enumerate}
	
\end{rmk}

\begin{proof}[Proof of Proposition \ref{prop.Jensen}]

The Jensen's formula applied on the unit disk says that if ${f}$ is a meromorphic function on a neighborhood of the unit disk $\overline{D}$ such that ${0}$ is neither a zero nor a pole of ${f}$, then
\[
\displaystyle \log |f(0)| = \int_0^1 \log |f(e^{2\pi i t})|\ dt + \sum_{\rho: |\rho| < 1} \log |\rho| \displaystyle - \sum_{\zeta: |\zeta| < 1} \log |\zeta|,
\]
where ${\rho}$ and ${\zeta}$ range over the zeros and poles of ${f}$ respectively (counting multiplicity). Upper bounding the above integral by the supremum norm on the circle $C$, we get
\[
\sum_{\rho} \log^-|\rho|- \sum_{\zeta} \log^-|\zeta| \displaystyle  \leq \log \|f\|_C - \log |f(0)|.
\]
Next, we apply it to $f=S\circ u^{-1}$. Then, the zeros and poles of $f$ are the images by $u$ of zeros and poles of $S$, where we count $\infty$ as a simple zero of $S$. So
\[
\sum_{\rho} \log^-|u(\rho)|-\sum_{\zeta} \log^-|u(\zeta)| \displaystyle  \leq \log \|S\|_{u^{-1}(C)} - \log |S(u^{-1}(0))|,
\]
where ${\rho}$ range over the zeros of ${S}$ and ${\zeta}$ range over its zeros. Now, note that if $P$ has simple roots, then the set of poles of $S=P'/P$ is precisely $\mathcal C$,  and its set of zeros is $\mathcal{Z}\cup\{\infty\}$. We then simply bound by below the term $\log^-|u(\infty)|$ by $0$ and we get the claimed inequality. In the case where $P$ has multiple roots, the two sums on the left hand side compensate and the inequality still holds.
\end{proof}

As announced above, the next step of the proof consists in applying Proposition \ref{prop.Jensen} to our random polynomial $P=P_n$ and make use of the almost sure estimates established in Section \ref{sec.estimates} to upper bound the right hand side in the corresponding Equation \eqref{eq.upjensen}. Before expliciting the resulting almost sure upper bound, let us note that the empirical measures $\nu_n$ are probability measure on $\mathbb C$, but can also be seen as probability measures on the Riemann sphere $\widehat{\mathbb C}=\mathbb C \cup \{\infty\}$. The space of probability measures on $\widehat{\mathbb C}$ being compact for the weak topology, we can thus assume that $\mathbb P$-almost surely, along a subsequence of integers, the sequence $(\nu_n)$ converges weakly to a cluster value $\widehat{\nu}_\infty$. Doing so, we obtain the following lemma.

\begin{lem}\label{lem.clusterprop}
Let us assume here that the base measure $\mu$ has not a finite support.
Almost surely with respect to $\mathbb P$, every cluster value $\widehat{\nu}_\infty$ of the sequence of probability measures $(\nu_n)$ for the weak topology on $\widehat{\mathbb C}$ satisfies the inequality
	$$\int_{\widehat{\mathbb{C}}}\log^-|u|d\widehat{\nu}_\infty \leq \int_{\mathbb{C}}\log^-|u|d\mu,$$
	for $\lambda_{\mathcal{M}}$-almost every $u$ in $\mathcal{M}$.
\end{lem}
\begin{proof}[Proof of Lemma \ref{lem.clusterprop}]
	First, let us notice that there exists a set of full probability, say $\Omega_1$ so that the following properties hold:	
for $\lambda_{\mathcal{M}}\otimes\lambda_{\mathcal{M}}\otimes\lambda_{\mathcal{M}}$-almost every triplet of Möbius transformations $(u_1,u_2,u_3)$:
	\begin{enumerate}
		\item[$i)$] $\log\|S_n\|_{u_i^{-1}(C)}=O(\log n)$ as $n\to +\infty$ for all indexes $i=1$, $2$ and $3$.
		\item[$ii)$] For $n$ large enough, $\log |S_n(u_i^{-1}(0))|\leq 0$ for at least one index $i=1$, $2$ or $3$.
		\item[$iii)$] $\int_{\mathbb C} \log^- |u_i| d\mu_n\to \int_{\mathbb C} \log^- |u_i| d\mu$ when $n\to +\infty$ for all indexes $i=1$, $2$ and $3$.
	\end{enumerate}	
Indeed, the first point is a consequence of Lemma \ref{lem.major} since for $\lambda_{\mathcal{M}}$-almost every $u$ in $\mathcal{M}$, $u^{-1}(C)$ is a circle $C(a,r)$ where $(a,r)$ belongs to the $\lambda_{\mathbb C}\otimes\lambda_{\mathbb R}$-full measure set where the conclusion of the lemma holds. 
	
The second point is a same way a consequence of Lemma \ref{lem.minor} (we use here that $\mu$ has not a finite support) since for $\lambda_{\mathcal{M}}\otimes\lambda_{\mathcal{M}}\otimes\lambda_{\mathcal{M}}$-almost every triplet $(u_1,u_2,u_3)$ in $\mathcal M$, $(u_1^{-1}(0),u_2^{-1}(0),u_3^{-1}(0))$ belongs to the $\lambda_{\mathbb C}\otimes\lambda_{\mathbb C}\otimes\lambda_{\mathbb C}$-full measure set of $\mathbb C^3$ where the conclusion of the lemma holds.
		
The third point is a straighforward consequence of the law of large numbers, since $\int_{\mathbb C} \log^- |u_i| d\mu_n$ is the sequence of empirical means of $Y_n=\log^- |u_i(Z_n)|$ and $\int_{\mathbb C} \log^- |u_i| d\mu$ is its expectation. Since $Y_n\geq 0$, it is not required that $Y_n$ is integrable.
\par
\smallskip
Next, we are going to prove that the conclusion of the statement holds for every event  $\omega\in \Omega_1$. So let us fix such an event $\omega$, and let us also fix  $\widehat{\nu}_\infty=\widehat{\nu}_\infty^{\omega}$ a cluster value of $(\nu_n^{\omega}$ on $\widehat{\mathbb C}$ for the weak topology.
\par
\smallskip
Let us fix for the moment $u_1, u_2, u_3$ in $\mathcal{M}$ so that the three properties stated above are satisfied.  Using points $i)$ and $ii)$ associated with Proposition \ref{prop.Jensen} with $P=P_n$, $u=u_1$, $u_2$ and $u_3$, we deduce that that for every integer $n$ there exists $i\in\{1,2,3\}$ such that
 $$\sum_{\rho} \log^-|u_i(\rho)| \leq \sum_{\zeta} \log^-|u_i(\zeta)| \displaystyle +O(\log n),$$
 where $\rho$ range over the critical points of $P_n$ and $\zeta$ range over its roots. We normalize the last inequality by dividing by $n$ and we get that for every integer $n$ there exists $i\in\{1,2,3\}$ such that
 \begin{equation}\label{ineq.distrib}
 \left(1-\frac{1}{n}\right)\int_{\mathbb C} \log^-|u_i|d\nu_n \leq \int_{\mathbb C} \log^-|u_i|d\mu_n+O\left(\frac{\log n}{n}\right).
  \end{equation}
By point $iii)$ above, letting $n$ go to infinity, we then get that for all $i\in\{1,2,3\}$
 \begin{equation}\label{eq.limmu}
  \lim_{n\to +\infty} \int_{\mathbb C} \log^- |u_i| d\mu_n=\int_{\mathbb C} \log^- |u_i| d\mu.
 \end{equation}
On the other hand, since $\widehat{\nu}_\infty$ is a cluster value of $(\nu_n)$, there exists an extraction $(n_k)$ such that  $\widehat{\nu}_\infty=\lim_{k\to +\infty} \nu_{n_k}$. Setting $\log_M^-(x)=\log^-(x) \wedge M$, the function $x \mapsto \log_M^-|u_i|$ is continuous on $\widehat{\mathbb C}$ for all $i\in\{1,2,3\}$ , so
  	$$
  	\int_{\widehat{\mathbb C}}\log^{-}_M|u_i|d\widehat{\nu}_\infty=\lim_{k \to +\infty} \int_{\mathbb C}  \log^{-}_M|u_i| d\nu_{n_k}\leq \limsup_{k \to +\infty} \int_{\mathbb C}  \log^{-}|u_i|d\nu_{n_k}.
  	$$
  	Letting then $M$ go to infinity, we obtain by monotone convergence
  	\begin{equation}\label{eq.limnu}
  	\int_{\widehat{\mathbb C}}\log^{-}|u_i|d\widehat{\nu}_\infty\leq \limsup_{k \to +\infty} \int_{\mathbb C}  \log^{-}|u_i|d\nu_{n_k}.
  	\end{equation}
Thus, letting $n$ tend to infinity along the subsequence $(n_k)_{k\in\mathbb N}$ in the inequality (\ref{ineq.distrib}), we deduce from \eqref{eq.limmu} and \eqref{eq.limnu} that there exists $i\in\{1,2,3\}$ such that
	$$\int_{\widehat{\mathbb C}}\log^{-}|u_i|d\widehat{\nu}_\infty\leq \int_{\mathbb C}\log^{-}|u_i|d\mu.$$
At this point, we obtained that for almost for $\lambda_{\mathcal{M}}\otimes\lambda_{\mathcal{M}}\otimes\lambda_{\mathcal{M}}$-almost every $u_1,u_2,u_3$ in $\mathcal{M}$,  there exists $i\in\{1,2,3\}$  such that $\int_{\widehat{\mathbb C}}\log^{-}|u_i|d\widehat{\nu}_\infty\leq \int_{\mathbb C}\log^{-}|u_i|d\mu$. It means that if we denote by $E$ the set
$$
E:=\left\{u\in\mathcal{M}\, \Big| \,\int_{\widehat{\mathbb C}}\log^{-}|u|d\widehat{\nu}_\infty> \int_{\mathbb C}\log^{-}|u|d\mu\right\},
$$
then the product set 
$$
E\times E\times E=\left\{(u_1,u_2,u_3)\in\mathcal{M}^3\, \big |\, \int_{\widehat{\mathbb C}}\log^{-}|u_i|d\widehat{\nu}_\infty> \int_{\mathbb C}\log^{-}|u_i|d\mu \text{ for all } i\in \{1,2,3\}\right\}
$$
 is $\lambda_{\mathcal{M}}\otimes\lambda_{\mathcal{M}}\otimes\lambda_{\mathcal{M}}$-negligible. As a result, we deduce that $E$ is $\lambda_{\mathcal{M}}$-negligible, hence the result.
\end{proof}

Finally, we want to verify that the conclusion of Lemma \ref{lem.clusterprop} is sufficient to ensure that $\widehat{\nu}_\infty=\mu$. That is the object of the two next lemmas.

\begin{lem}\label{lem.affine}
	Let $m_1$ and $m_2$ two positive finite measures on $\mathbb C$ so that 
	$$\int_{\mathbb{C}}\log^-|u|dm_1 \leq \int_{\mathbb{C}}\log^-|u|dm_2$$
	for $\lambda_{\mathcal A}$-almost every affine transformation $u$ of $\mathcal{A}$. Then $m_1\leq m_2$ on $\mathbb C$. In particular, if $m_1$ and $m_2$ are probability measures then $m_1=m_2$.
\end{lem}

\begin{proof}
Take $\varphi$ a non-negative, continuous bounded function on $\mathbb C$. For $\lambda_{\mathbb R}$-almost every $r>0$ and for $\lambda_{\mathbb C}$-almost every $z\in\mathbb{C}$, by using the assumption with $u(w)=\frac{w-z}{r}$, we have
\[
\varphi(z) \int_{\mathbb C}  \log^{-}\left( \frac{|w-z|}{r} \right)dm_1(w) \leq \varphi(z) \int_{\mathbb C}  \log^{-}\left( \frac{|w-z|}{r} \right)dm_2(w).
\]
Integrating over $\mathbb C$ we get that for $\lambda_{\mathbb R}$-almost every $r>0$,
\[
\iint_{\mathbb C^2} \varphi(z)  \log^{-}\left( \frac{|w-z|}{r} \right)dm_1(w) dz \leq \iint_{\mathbb C^2}  \varphi(z)  \log^{-}\left( \frac{|w-z|}{r} \right)dm_2(w) dz.
\]
Performing the change of variables $z= w+rz'$ and dividing by $r$ we get
\[
\int_{\mathbb C} \left( \int_{\mathbb C} \varphi(w+rz')dm_1(w) \right) \log^-(|z'|)dz' \leq \int_{\mathbb C} \left( \int_{\mathbb C} \varphi(w+rz')dm_2(w)\right) \log^-(|z'|)dz'.
\]
Since $\int_{\mathbb C} \log^-(|z'|)dz'<+\infty$, letting $r$ go to zero, by dominated convergence, we obtain after simplification
\[
\int_{\mathbb C}  \varphi(w)dm_1(w)  \leq \int_{\mathbb C} \varphi(w)dm_2(w).
\]
This implies the result.
\end{proof}

The above lemma  would be sufficient to conclude if we knew in advance that the sequence $(\nu_n)$ is tight, seen as a sequence of measures on $\mathbb C$, to ensure that its cluster values on $\mathbb C$ are probability measures. To conclude in our case where the measures $(\nu_n)$ are indeed tight, but seen as measures on $\widehat{\mathbb C}$, we thus need the following variation of the last Lemma \ref{lem.affine} on the Riemann sphere.

\begin{lem}\label{lem.homo}
	Let $\widehat{m_1}$ and $\widehat{m_2}$ two finite measures on $\widehat{\mathbb C}$ so that 
	$$\int_{\widehat{\mathbb{C}}}\log^-|u|d\widehat{m_1} \leq \int_{\widehat{\mathbb{C}}}\log^-|u|d\widehat{m_2}$$
	for $\lambda_{\mathcal M}$-almost every Möbius transformation $u$ in $\mathcal{M}$. Then $\widehat{m_1}\leq \widehat{m_2}$ on $\widehat{\mathbb C}$. In particular, if $\widehat{m_1}$ and $\widehat{m_2}$ are probability measures then $\widehat{m_1}=\widehat{m_2}$.
\end{lem}

\begin{proof}
We use here the fact that if $u$ is a $\lambda_{\mathcal A}$-generic element of $\mathcal A$ and if $v$ is a $\lambda_{\mathcal M}$-generic element of $\mathcal M$, then $u \circ v$ is a 	$\lambda_{\mathcal M}$-generic element of $\mathcal M$. The assumption implies that for $\lambda_{\mathcal A}$-almost every $u$ in $\mathcal{A}$ and  $\lambda_{\mathcal M}$-almost every $v$ in $\mathcal{M}$
		$$\int_{\widehat{\mathbb{C}}}\log^-|u\circ v|d\widehat{m_1} \leq \int_{\widehat{\mathbb{C}}}\log^-|u\circ v|d\widehat{m_2}.$$
	Moreover, for $\lambda_{\mathcal M}$-almost every $v$ in $\mathcal{M}$, $v^{-1}(\infty)$ is not an atom of $\widehat{m_1}$ and $\widehat{m_2}$. Then the image measures  $m_1=v_*\widehat{m_1}$ and $m_2=v_*\widehat{m_2}$ have no atoms at $\infty$, and then satisfy 
	$$\int_{\mathbb{C}}\log^-|u|dm_1 \leq \int_{\mathbb{C}}\log^-|u|dm_2$$
	for $\lambda_{\mathcal A}$-almost every $u$ in $\mathcal{A}$. We deduce from Lemma \ref{lem.affine} that $m_1\leq m_2$ on $\mathbb C$, and so on $\widehat{\mathbb{C}}$.
Thus, for $\lambda_{\mathcal M}$-almost every $v$ in $\mathcal{M}$, $v_*\widehat{m_1}\leq v_*\widehat{m_2}$. In particular, there exist a sequence of elements $v_k$ of $\mathcal{M}$ such that $(v_k)*\widehat{m_1}\leq (v_k)_*\widehat{m_2}$ and converging uniformly to the identity mapping on $\widehat{\mathbb C}$ when $k$ goes to infinity. Thus $\widehat{m_1}\leq \widehat{m_2}$.
	
\end{proof}

We can now complete the proof of Theorem \ref{thm.main} in the case where the base measure has no finite support. Combining the conclusions of Lemma \ref{lem.clusterprop}  and Lemma \ref{lem.homo}, there exists a set of $\mathbb P$-full measure such that, for any cluster value of the sequence $(\nu_n)=(\nu_n^{\omega})$, seen as measures on $\widehat{\mathbb C}$, we get that $\widehat{\nu}_\infty= \mu$. As a result, $\mu$ is the only cluster value of the sequence $(\nu_n)$, and the latter converges to $\mu$, which is in fact a probability measure on $\mathbb C$, hence the result.

\par
\medskip
\noindent
Let us conclude the article by a few remarks and possible natural extensions of the results presented here. 
\begin{enumerate}
	\item In the recent reference \cite{MR4474893}, Byun, Lee and Reddy proved that Theorem \ref{thm.kab} remains true if we replace the derivative of $P_n$ by higher derivatives. It seems reasonable to think that the almost sure convergence in Theorem \ref{thm.main} extends to a finite number of derivatives. 
	\item Having established the almost sure convergence of both sequences $(\mu_n)$ and $(\nu_n)$ to $\mu$, it is natural to try to quantify the rate of convergence to $0$ of the distance $d(\mu_n, \nu_n)$ for some natural metric $d$ (e.g. Prorohov, Kolmogorov, Wasserstein...). This also open the doors to the study of the fluctuations of the difference $\mu_n-\nu_n$.
	\item Can we obtain analog results for more general sequences of random polynomials $P_n$ than this model where the roots are chosen i.i.d. ? This question is broad and has already been studied in specific contexts in some cited papers here, e.g. \cite{MR3698743,MR3940764}.
\end{enumerate}

 \if{

\begin{rmk}
	Here, one can moreover make explicit the fluctuations around the limit. Indeed, we have 
	\[
	\nu_n-\mu = \sum_{i=1}^r \left( \frac{N_i-1}{n-1} -p_i\right) \delta_{z_i} + \frac{1}{n-1} \sum_{i=1}^{r-1} \delta_{z_i'}.
	\]
	By the Central Limit Theorem, we have 
	\[
	\sqrt{n} \left(\left( \frac{N_1}{n}, \ldots, \frac{N_r}{n} \right) - (p_1, \ldots, p_r)\right)\xrightarrow[n \to +\infty]{d} \mathcal N_r(0, K),
	\]
	where the covariance matrix $K=(K_{ij})$ is given by $K_{i,j}:= \delta_{ij} p_i - p_i p_j$. Now, if $\varphi$ is a bounded measurable test function on $\mathbb C$, and if we set $\bar \varphi:=(\varphi(z_1), \ldots, \varphi(z_r))$ we can decompose
	\[
	\begin{array}{ll}
	\displaystyle{\sqrt{n} \left( \nu_n(\varphi)-\mu(\varphi)\right) }& = \displaystyle{\sum_{i=1}^r \sqrt{n} \left( \frac{N_i-1}{n-1} -p_i\right) \varphi(z_i) + \frac{\sqrt{n}}{n-1} \sum_{i=1}^{r-1} \varphi(z_i')}\\
	\\
	& =\displaystyle{\frac{n}{n-1}\sum_{i=1}^r \sqrt{n} \left( \frac{N_i}{n} -p_i\right) \varphi(z_i) + O\left(\frac{1}{\sqrt{n}}  \right).}
	\end{array}
	\]
	Therefore, we get 
	\[
	\sqrt{n} \left( \nu_n(\varphi)-\mu(\varphi)\right)  \xrightarrow[n \to +\infty]{d}\mathcal N_1 \left(0, \bar \varphi K \bar \varphi^t\right),
	\]
	where the limit variance can be rewritten as 
	\[
	\bar \varphi K \bar \varphi^t = \sum_{i,j=1}^r \varphi(z_i)\varphi(z_j)(\delta_{ij} p_i - p_i p_j) = \sum_{i=1}^r \varphi(z_i)^2 p_i - \left(  \sum_{i=1}^r\varphi(z_i) p_i\right)^2
	\]
	or in other words
	\[
	\bar \varphi K \bar \varphi^t=\int \phi(z)^2 \mu(dz) - \left(\int \phi(z) \mu(dz) \right)^2= \text{var}_{\mu}(\phi).
	\]
\end{rmk}
}\fi

\if{
{\small 
\bibliographystyle{alpha}
\bibliography{biblio}
}
}\fi
\end{document}